\documentclass{amsart}
\usepackage{amsmath,amssymb}
\usepackage{amscd}
\usepackage{graphicx}
\usepackage[all]{xy}
\newtheorem{theorem}{Theorem}[section]

\newtheorem{lemma}{Lemma}[section]

\newtheorem{proposition}{Proposition}[section]
\newtheorem{main theorem}{Main Theorem}[section]

\newcommand{\pa}{\partial}

\newcommand{\nn}{\nonumber}

\title[continuously non-extendable mappings]
{continuously non-extendable mappings between generalized complex ellipsoids of 
different dimensions}

\author{Atsushi Hayashimoto}
\address{Atsushi Hayashimoto: 716 Tokuma, Nagano 381-8550, Japan}
\email{atsushi@nagano-nct.ac.jp}


\begin{document}

\maketitle
\begin{abstract}
There exists a proper holomorphic mapping 
between balls of different dimensions such that 
it does not extend continuously to the boundary. 
The aim of this paper is to show the same phenomenon occurs for 
pseudoconvex domains of different dimensions. 
\end{abstract}

\setcounter{footnote}{-1}
\footnote
{This work was supported by Grant-in-Aid for Scientific Research (C) 17K05308 
and 22K03364}

\section{Introduction} 
Let $D_1, D_2$ be domains in complex Euclidean spaces and 
$F:D_1 \to D_2$ a proper holomorphic mapping. 
If both domains are of the same dimensions, 
then, under the assumptions of pseudoconvexity and 
regularity of the boundaries, the mapping extends to the closure of the source domain 
as a $C^k$ mapping.  
These kinds of theorems were proved by e.g. Diederich-Fornaess \cite{DF}, Bell-Catlin \cite{BC}. 
In contrast to the equi-dimensional case,  
the higher codimensional case, which means that the target dimension is 
bigger than the source one, has not been studied much yet 
and some known results state that, under certain conditions on the domains, 
there exists a proper holomorphic mapping 
between them such that it can not be extended $C^k$ manner across the boundary. 
For example, F.~Forstneric \cite{F} proved that
\begin{theorem}
For each integer $n \geq 1$, there is a proper holomorphic 
embedding $F : B^n \to B^N, N=n+1+2s$, where $s=s(n)$ is determined by $n$, such that $F$ 
does not extend continuously to $\overline{B^n}$. 
\end{theorem}
While known non-extendability theorems, e.g. \cite{D} and \cite{G}, treat the balls or discs,   
we treat, in this paper, the case of pseudoconvex domains, especially, generalized complex 
pseudoellipsoids. 
As far as the author's knowledge, this is the first 
non-extendability theorem for pseudoconvex domains.  
The difficulty to study generalized complex pseudoellipsoids comes from the fact that 
its defining function does not define a norm, and therefore, 
we can not use the triangle inequality. 

For integers $m_1, \dots, m_{k}$, 
we let $M_0=0$ and $M_k=m_1+\dots+m_k$. 
We use notation $Z=(z_{(1)}, \dots, z_{(s+1)}) 
\in \mathbb{C}^{m_1} \times \dots \times \mathbb{C}^{m_{s+1}}$
and $z_{(k)}=(z^{M_{k-1}+1}, \dots, z^{M_{k}}) \in \mathbb{C}^{m_{k}}$. 
Let $E(m_1, \dots, m_{s+1}; \alpha_1, \dots, \alpha_s)$ be a domain with real analytic boundary 
defined by 
\begin{align}
&E(m_1, \dots, m_{s+1}; \alpha_1, \dots, \alpha_s)  \nn \\
&=
\{Z \in \mathbb{C}^{M_{s+1}} : 
||z_{(1)}||^{2\alpha_1}+\dots+||z_{(s)}||^{2\alpha_s}+||z_{(s+1)}||^2-1<0,  
\alpha_j \in \mathbb{N}\}. \nn
\end{align}
This domain is called a generalized complex pseudoellipsoid. 
For simplicity, we write 
$E(m_1, \dots, m_{s+1}; \alpha_1, \dots, \alpha_s)=E((m);(\alpha))$ and 
$||z_{(1)}||^{2\alpha_1}+\dots+||z_{(s)}||^{2\alpha_s}+||z_{(s+1)}||^2=|||Z|||^{2\alpha}$. 
Let $E((n); (\beta))$ be another generalized complex pseudoellipsoid with the same dimension as 
$E((m);(\alpha))$. 
If there exists a proper holomorphic mapping $F : E((m);(\alpha)) \to E((n);(\beta))$, 
then it extends holomorphically to $\overline{E((m);(\alpha))}$. 
The aim of the present paper is to show the non-extension phenomenon 
when the dimension differs. 
In the higher codimensional case, 
there exists a proper holomorphic mapping which can not be extended 
continuously to the closure of the source domain. 
The main theorem is the following. 

\begin{main theorem}\label{mth}
For a generalized complex pseudoellipsoid $E((m); (\alpha))$, 
there exists a generalized complex pseudoellipsoid 
$E((n); (\beta))$ and a proper holomorphic mapping 
$F : E((m); (\alpha)) \to E((n); (\beta))$ such that $F$ does not extend continuously 
to $\overline{E((m); (\alpha))}$. 
\end{main theorem}

The purpose of the existence of non-extendability theorem is the following.  
Let $D_1 \subset \mathbb{C}^{m} , D_2 \subset \mathbb{C}^n, m<n,$ and $F$ 
be the same as in the begining of this section. 
We want to know the conditions for $D_1, D_2$ and $F$ to guarantee 
the holomorphic extendability across the boundary.   
X.~Huang \cite{H} proved the following theorem. 
\begin{theorem}
Let $f:B^m \to B^n$ be a proper holomorphic mapping 
which extends as a $C^2$ mapping across the boundary. 
Assume that $m \leq n \leq 2m-2$. 
Then $f$ is of the form $f(z)=(z, 0, \dots, 0)$ up to 
the automorphisms of balls,  
where $0$ is added $n-m$ components.   
\end{theorem}
This theorem asserts that $C^2$ extendability to the closure of 
the ball implies the holomorphic extendability to $\mathbb{C}^m$. 
Therefore, it is important to know whether the mapping under 
consideration is $C^2$-extendable or not. 
In this view point, A.~Dor \cite{D} proved that 
there exists a proper holomorphic mapping $f:B^m \to B^{m+1}$ 
such that $f \in C^0(\overline{B^m})$ but $f \not\in C^2(\overline{B^m})$. 
On the other hand, originally, 
we do not know the conditions on the proper holomorphic mapping 
$f:B^m \to B^n, m+1 \leq n \leq 2m-2$,  
which guarantee the $C^0$ or the $C^2$ extendability. 
Therefore, first of all, it is important to know whether 
the mappings have the $C^0$ extension and 
if do not, then, to know the conditions 
on the $C^0$-extendability. 
Our main theorem is one of the answer to this question in the case of 
generalized pseudoellipsoids. 

The organization of this article is the following. 
In the section~2, we define a peaking function on the closure 
of the generalized pseudoellipsoid  
and obtain its estimates. 
In the section~3, we construct a entire mapping depending on 
a holomorphic mapping which satisfies certain estimates.  
Next, we construct a sequences of entire mappings 
which lead to the desired proper holomorphic mapping. 
Finally, we construct a desired proper holomorphic mapping depending on 
the harmonic function on the disc which is not continuous 
on the boundary. 
 
\section{some estimates}
In what follows, we fix $m_1, \dots, m_{s+1} \in \mathbb{N}$ 
and $\alpha_1, \dots, \alpha_{s} \in \mathbb{N}_{\geq 2}$.   
For $a \in \mathbb{C}^n$ and $r>0$, we denote by $B(a, r)$ the open ball of radius $r$ and 
centered at $a$. 
Take $W_i \in \pa E((m); (\alpha)), 1 \leq i \leq I$, where 
the number I is chosen sufficiently large so that, 
for fixed $\lambda$ and sufficiently small fixed $r$,  
$\pa E((m); (\alpha)) \subset \bigcup B(W_i, \lambda r)$ holds. 
$\lambda$ and $r$ will be determined later. 
We assume that $I$ is sufficiently large and divide it into $I=n_1+\dots+n_{t+1}$. 
Integers $n_j \geq 2$ are specified bellow depending on $s$. 
As above, we use the notation $N_0=0$ and $N_k=n_1+\dots+n_k$. 
We make copies of $W_i$ such as 
$W_{i+n_{k}}=W_i$ 
for $2N_{k-1}+1 \leq i \leq 2N_{k-1}+n_{k}$. 

Let $E(n_1, \dots, n_{t+1}; \beta_1, \dots, \beta_t)$, for shortly $E((n); (\beta))$,  
be a domain with real analytic boundary 
defined by 
\begin{align}
&E((n);(\beta)) \nn \\
& \quad 
=\{Z=(z_{(1)}, \dots, z_{(t+1)}) 
\in \mathbb{C}^{n_1} \times \dots \times \mathbb{C}^{n_{t+1}+p} : 
|||Z|||^{2\beta}-1<0,  
\beta_j \in \mathbb{N}\}.  \nn
\end{align}
We take $n_1, \dots, n_{t+1}, p$ and $\beta_1, \dots, \beta_{t}$ 
so as to exist a proper holomorphic mapping between 
$E((m); (\alpha))$ and $E((n); (\beta))$. 

For a real analytic function
$\rho(Z)$ and a vector  
$v \in \mathbb{C}^{M_{s+1}}$, 
we define its gradient and real Hessian as 
\begin{equation}
N(Z)=\left(\dfrac{\pa \rho}{\; \pa \bar{z}^1 \;}, \dots, 
\dfrac{\pa \rho}{\; \pa \bar{z}^{M_{s+1}} \; } \right),  \nn
\end{equation}
and 
\begin{equation}
Q_Z(v)=\mathrm{Re}\left(\sum^{M_{s+1}}_{\gamma, \delta=1}
\dfrac{\pa^2\rho}{\; \pa z^{\gamma}z^{\delta} \;}(Z, \bar{Z})v^{\gamma}v^{\delta}\right)
+\sum^{M_{s+1}}_{\gamma, \delta=1}\dfrac{\pa^2 \rho}
{\; \pa z^{\gamma} \pa \bar{z}^{\delta} \;}(Z, \bar{Z})v^{\gamma} \bar{v}^{\delta}. \nn
\end{equation}

If $\rho(Z)=||z_{(1)}||^{2\alpha_1}+\dots+||z_{(s)}||^{2\alpha_s}+||z_{(s+1)}||^2-1$, 
then we have 
\begin{align}
&\langle W_{i}-Z, N(W_{i}) \rangle \nn \\
&=\alpha_1||w_{i(1)}||^{2\alpha_1}+\dots+
\alpha_s||w_{i(s)}||^{2\alpha_s}+||w_{i(s+1)}||^2   \nn \\
&\qquad -\alpha_1||w_{i(1)}||^{2(\alpha_1-1)}\langle z_{(1)}, w_{i(1)} \rangle-\dots-
\alpha_s||w_{i(s)}||^{2(\alpha_s-1)}\langle z_{(s)}, w_{i(s)} \rangle   \nn \\
&\qquad -\langle z_{(s+1)}, w_{i(s+1)} \rangle.  \nn 
\end{align}

Put $h(t)=\rho((1-t)W_i+tZ), \; t \in \mathbb{R}$ for $W_i, Z \in \pa E((m);(\alpha))$. 
Then $h(0)=h(1)=0$. 
By the chain rule, 
\begin{align}
h'(0)&=2\mathrm{Re}\langle Z-W_i, N(W_i) \rangle,  \nn \\
h''(t)&=2Q_{(1-t)W_i+tZ}(W_i-Z).   \nn
\end{align}
Insert these data into the Taylor formula 
$h(t)=h(0)+h'(0)t+1/2h''(t_0)t^2$ 
and put $t=1$, 
then we have 
\begin{equation}
2\mathrm{Re}\langle W_i-Z, N(W_i) \rangle=Q_{(1-t_0)W_i+t_0Z}(W_i-Z)   \nn
\end{equation}
for some $t_0 \in (0, 1)$. 

In order to estimate $\text{Re}\langle W_i-Z, N(W_i)/||N(W_i)||\rangle$ in terms of $||W_i-Z||$, 
we need to find $Z$ such that the quadratic form $Q_{(1-t_0)W_i+t_0Z}(W_i-Z)$ degenerates,  
namely, $Z$ such that $Q_{(1-t_0)W_i+t_0Z}(W_i-Z)=0$ for $W_i-Z \not=0$. 
Assume $W_i\not=Z$. 
Since it is calculated as 
\begin{align}
&Q_{(1-t_0)W_i+t_0Z}(W_i-Z)  \label{Q}\\
&=\sum^{s}_{k=1}\Biggl[(\alpha_k-1)
\Bigl\{\text{Re}\sum^{M_k}_{\gamma=M_{k-1}+1}
((1-t_0)\bar{w}^{\gamma}_{i}+t_0\bar{z}^{\gamma})(w^{\gamma}_{i}-z^{\gamma}) \Bigr\}^2
\nn \\
&\qquad \qquad \qquad \qquad \qquad 
+||(1-t_0)w_{i(k)}+t_0z_{k}||^2\sum^{M_k}_{\gamma=M_{k-1}+1}|w^{\gamma}_{i}-z^{\gamma}|^2 \Biggr]   \nn \\
&\qquad \qquad \qquad \qquad 
\times \alpha_k||(1-t_0)w_{i(k)}+t_0z_{(k)}||^{2(\alpha_k-2)} \nn\\
&\qquad 
+\sum^{M_{s+1}}_{\gamma=M_{s}+1}|w^{\gamma}_{i}-z^{\gamma}|^2=0,   \nn
\end{align}
we obtain $w_{i(s+1)}=z_{(s+1)}$ and there exists at least one $k$ such that 
$(1-t_0)w_{i(k)}+t_0z_{(k)}=0$. 
Assume that, by reordering the variables, 
$(1-t_0)w_{i(1)}+t_0z_{(1)}\not=0, \dots, (1-t_0)w_{i(\ell)}+t_0z_{(\ell)}\not=0$ and 
$(1-t_0)w_{i(\ell+1)}+t_0z_{(\ell+1)}=0, \dots, (1-t_0)w_{i(s)}+t_0z_{(s)}=0$. 
Then it follows from (\ref{Q}) that $w_{i(1)}=z_{(1)}, \dots, w_{i(\ell)}=z_{(\ell)}$. 
Summary, the point $Z$ where $Q_{(1-t_0)W_i+t_0Z}(W_i-Z)$ degenerates is written as 
\begin{equation}
Z=(w_{i(1)}, \dots, w_{i(\ell)}, \dfrac{\; t_0-1 \;}{t_0}w_{i(\ell+1)}, \dots, 
\dfrac{\; t_0-1 \;}{t_0}w_{i(s)}, w_{i(s+1)}). \label{zdeg}
\end{equation}
Since the points $Z$ and $W_i$ are in the boundary of $E((m);(\alpha))$, the identity 
\begin{align}
&||w_{i(1)}||^{2\alpha_1}+\dots+||w_{i(\ell)}||^{2\alpha_{\ell}}+
||\dfrac{\; t_0-1 \;}{t_0}w_{i(\ell+1)}||^{2\alpha_{\ell+1}}  
+\dots   \label{wt} \\
&+||\dfrac{\; t_0-1 \;}{t_0}w_{i(s)}||^{2\alpha_{s}}   
+(1-||w_{i(1)}||^{2\alpha_1}-\dots-||w_{i(s)}||^{2\alpha_{s}})-1=0 \nn 
\end{align}
holds. 
Therefore 
if $Z$ is of the form (\ref{zdeg}) for fixed $W_i$ and for some $t_0$ satisfying (\ref{wt}), 
then $Q_{(1-t_0)W_i+t_0Z}(W_i-Z)$ degenerates. 
Now we estimate $\text{Re}\langle W_i-Z, N(W_i)/||N(W_i)||\rangle$ in terms of $||W_i-Z||$. 
\begin{lemma}\label{Reestimate}
There exist positive constants $A_1, A_2, B_1,B_2$ such that the following estimates hold. 

$\mathrm{(i)}$ If $Z \in \pa E((m);(\alpha))$ is not of the form $\mathit(\ref{zdeg})$, then 
\begin{equation}
A_1||W_i-Z||^2\leq \mathrm{Re}\langle W_i-Z, N(W_i)/||N(W_i)||\rangle
\leq A_2||W_i-Z||^2  \nn
\end{equation}
holds. 

$\mathrm{(ii)}$  
If $Z \in \pa E((m);(\alpha))$ is of the form $(\ref{zdeg})$, then  
\begin{equation}
B_1||W_i-Z||^{2\alpha_{\mathrm{max}}}\leq \mathrm{Re}\langle W_i-Z, N(W_i)/||N(W_i)||\rangle 
\leq B_2||W_i-Z||^{2\alpha_{\mathrm{min}}}   \nn
\end{equation}
holds. 
Here we denote by $\alpha_{\mathrm{max}}=\mathrm{max}\{\alpha_1, \dots, \alpha_s\}$ and 
$\alpha_{\mathrm{min}}=\mathrm{min}\{\alpha_1, \dots, \alpha_s\}$. 
\end{lemma}

\begin{proof}
Proof (i). This is the case that $Q_{(1-t_0)W_i+t_0Z}(W_i-Z)$ does not degenerate, 
which is the same situation as in 5.1~Lemma in \cite{F}. 

Proof (ii). 
In this case, 
it suffices to show that 
\begin{equation}
\dfrac{\; \text{Re}\langle W_i-Z, N(W_i)/||N(W_i)|| \rangle \;}{||W_i-Z||^{2\alpha_{\text{max}}}} 
\geq B_1  \nn
\end{equation}
and 
\begin{equation}
\dfrac{\; \text{Re}\langle W_i-Z, N(W_i)/||N(W_i)|| \rangle \;}{||W_i-Z||^{2\alpha_{\text{min}}}} 
\leq B_2.  \nn
\end{equation} 
$\text{Re}\langle W_i-Z, N(W_i)/||N(W_i)|| \rangle$ is computed as follows. 
\begin{align}
&\mathrm{Re}\langle W_i-Z, N(W_i)/||N(W_i)|| \rangle  \nn \\
&=\dfrac{\; \alpha_{\ell+1}||w_{i(\ell+1)}||^{2\alpha_{\ell+1}}+\dots+\alpha_{s}||w_{i(s)}||^{2\alpha_{s}}  \;}
{\; t_0\sqrt{\alpha_1^2||w_{i(1)}||^{4\alpha_1-2}+\dots+\alpha_s^2||w_{i(s)}||^{4\alpha_s-2}+1
-||w_{i(1)}||^{2\alpha_1}-\dots-||w_{i(s)}||^{2\alpha_{s}}} \;}.     \nn 
 \end{align}  
The inequalities 
\begin{align}
&\dfrac{\; \alpha_{\ell+1}||w_{i(\ell+1)}||^{2\alpha_{\ell+1}}+\dots+\alpha_{s}||w_{i(s)}||^{2\alpha_{s}} \;}
{\; ||W_i-Z||^{2\alpha_{\mathrm{max}}} \;}   \nn \\
&>
\dfrac{\; t_0^{2\alpha_{\mathrm{max}}}\alpha_{\mathrm{min}}
\left\{||w_{i(\ell+1)}||^{2\alpha_{\mathrm{max}}}+\dots
+||w_{i(s)}||^{2\alpha_{\mathrm{max}}}\right\} \;}
{\; (s-\ell)^{\alpha_{\mathrm{max}}-1}
\left\{||w_{i(\ell+1)}||^{2\alpha_{\mathrm{max}}}+\dots
+||w_{i(s)}||^{2\alpha_{\mathrm{max}}}\right\} \;}    \nn  \\
&=
\dfrac{\; t_0^{2\alpha_{\mathrm{max}}}\alpha_{\mathrm{min}} \;}
{\; (s-\ell)^{\alpha_{\mathrm{max}}-1} \;}  \nn
\end{align}
and 
\begin{align}
&\alpha_1^2||w_{i(1)}||^{4\alpha_1-2}+\dots+\alpha_s^2||w_{i(s)}||^{4\alpha_s-2}
+1-||w_{i(1)}||^{2\alpha_1}-\dots-||w_{i(s)}||^{2\alpha_{s}}  \nn   \\
&<\alpha_1^2+\dots+\alpha_s^2+1  \nn 
\end{align}
imply that $\text{Re}\langle Z-W_i, N(W_i)/||N(W_i)||\rangle /||W_i-Z||^{2\alpha_{\text{max}}}$ 
is bigger than 
\begin{equation}
B_1=\dfrac{\; t^{2\alpha_{\mathrm{max}}-1}_{0}\alpha_{\text{min}} \;}{\; (s-\ell)^{\alpha_{\mathrm{max}}-1} 
\sqrt{\alpha_1^2+\dots+\alpha_s^2+1} \;}.   \nn
\end{equation}
The inequalities 
\begin{align}
&\dfrac{\; \alpha_{\ell+1}||w_{i(\ell+1)}||^{2\alpha_{\ell+1}}+\dots+\alpha_{s}||w_{i(s)}||^{2\alpha_{s}} \;}
{\; ||W_i-Z||^{2\alpha_{\mathrm{min}}} \;}    \nn \\
&<
\dfrac{\; t^{2\alpha_{\mathrm{min}}}_{0}\alpha_{\mathrm{max}}\left\{||w_{i(\ell+1)}||^{2\alpha_{\mathrm{min}}}
+\dots
+||w_{i(s)}||^{2\alpha_{\mathrm{min}}}\right\} \;}
{\; ||w_{i(\ell+1)}||^{2\alpha_{\mathrm{min}}}+\dots
+||w_{i(s)}||^{2\alpha_{\mathrm{min}}} \;}    \nn \\
&=t^{2\alpha_{\mathrm{min}}}_{0}\alpha_{\mathrm{max}}.     \nn
\end{align}
and 
\begin{align}
&\alpha_1^2||w_{i(1)}||^{4\alpha_1-2}+\dots+\alpha_s^2||w_{i(s)}||^{4\alpha_s-2}   \nn \\
& \qquad \qquad \qquad \qquad  
+1-||w_{i(1)}||^{2\alpha_1}-\dots-||w_{i(s)}||^{2\alpha_{s}}>\dfrac{1}{\; 2 \;}   \nn 
\end{align}
if $0 \leq ||w_{i(1)}||^{2\alpha_1}+\dots+||w_{i(s)}||^{2\alpha_{s}} \leq 1$ 
show that we can take 
$B_2=\sqrt{2}\alpha_{\text{max}}t^{2\alpha_{\mathrm{min}}-1}_{0}$. 
\end{proof}
For $\nu>0$, the function 
\begin{equation}
\phi_{i}(Z)=e^{-\nu \langle W_i-Z, N(W_i)/||N(W_i)|| \rangle}   \nn
\end{equation}
is a peaking function on $\overline{E((m);(\alpha))}$ and 
it follows from Lemma~\ref{Reestimate} that its modulus 
\begin{equation}
|\phi_{i}(Z)|=e^{-\nu\mathrm{Re} \langle W_i-Z, N(W_i)/||N(W_i)|| \rangle}   \nn
\end{equation}
is estimated as 
\begin{equation}
e^{-\nu A||W_i-Z||^{2p}} \leq |\phi_{i}(Z)| \leq e^{-\nu B||W_i-Z||^{2q}},   \label{phiest}
\end{equation}
for some $p, q, A$ and $B$. 
Note that $\phi_{i}=\phi_{i+n_{k}}$ for each $i=2N_{k-1}+1, 
\dots, 2N_{k-1}+n_{k}, k=1. \dots, t+1$. 
Let $g^i$ be an entire function of the form 
\begin{equation}
g^{i}(Z)=\gamma_{i}\phi_{i}(Z), \quad |\gamma_{i}| \leq 1.  \label{gi}
\end{equation}
Let $A_1, A_2, B_1, B_2$ be as in Lemma~{\ref{Reestimate}}. 
We can choose $r$ such that $0 < r <1$ and 
\begin{equation}
\left(\dfrac{\; 4\beta^{2}_{t}B_2r^{2(\alpha_{\mathrm{min}}-\alpha_{\mathrm{max}})} \;}{B_1}
\right)^{1/(2\alpha_{\mathrm{max}})}
\geq
\left(\dfrac{\; 4\beta^{2}_{t}A_2 \;}{A_1}\right)^{1/2}.  \label{aabb}
\end{equation}
Define $\lambda$ as 
\begin{equation}
\lambda=\left(\dfrac{\; 4\beta^{2}_{t}B_2r^{2(\alpha_{\mathrm{min}}-\alpha_{\mathrm{max}})} \;}{B_1}
\right)^{1/(2\alpha_{\mathrm{max}})}.   \nn  
\end{equation}

\begin{lemma}\label{gestimate}
For each sufficiently small $\eta>0$, there exist $\nu>0$ and $r \; (0<r<1)$ such that 
$r$ satisfies $\mathrm{(\ref{aabb})}$ 
and the following statement holds, 

If $Z \in \pa E((m);(\alpha))$ satisfies $||W_i-Z||>\lambda r$, 
then  
$|g^i(Z)|<\eta$. 
\end{lemma}
\begin{proof}
First, we take sufficiently small $\eta>0$. 
We divide the proof into two parts according to the form of $Z$. 

(I) $Z$ is not of the form $\mathit(\ref{zdeg})$. 

By Lemma~\ref{Reestimate} (i) and 
$\lambda \geq \left(4\beta^{2}_{t}A_2/A_1\right)^{1/2}$,  
the following inequalities hold, 
\begin{equation}
-\nu\text{Re} \langle W_i-Z, N(W_i)/||N(W_i)|| \rangle \leq -\nu A_1||W_i-Z||^2 
<-4\nu A_2r^2\beta^{2}_{t}.  \nn   
\end{equation}
Therefore we have 
\begin{equation}
|g^i(Z)|<e^{-\nu\text{Re} \langle W_i-Z, N(W_i)/||N(W_i)|| \rangle}<e^{-4\nu A_2r^2\beta^{2}_{t}}.    \nn
\end{equation}
For any given $\eta$, we can choose $\nu$ and $r$  
such that 
\begin{equation}
\nu r^2>\dfrac{1}{\; 4A_2\beta^{2}_{t} \;}\log\dfrac{1}{\; \eta \;}.  \nn
\end{equation}

(II) $Z$ is of the form $\mathit(\ref{zdeg})$. 

By Lemma~\ref{Reestimate} (ii) and 
by the definition of $\lambda$, 
the same argument as in (I) implies 
\begin{equation}
|g^i(Z)|<e^{-4\nu B_2r^{2\alpha_{\mathrm{min}}}\beta^{2}_{t}}.  \nn
\end{equation}
For any given $\eta$, we can choose $\nu$ and $r$  
such that 
\begin{equation}
\nu r^{2\alpha_{\mathrm{min}}}>
\dfrac{1}{\; 4B_2\beta^{2}_{t} \;}\log\dfrac{1}{\; \eta \;}.   \nn
\end{equation}
Now if we choose $\nu$ and $r$ to satisfy 
\begin{equation}
\nu r^{2\alpha_{\mathrm{min}}}>
\dfrac{1}{\; 4\text{min}\{A_2, B_2\}\beta^{2}_{t} \;}\log\dfrac{1}{\; \eta \;},  \nn
\end{equation}
then they satisfy all conditions posed in this lemma. 
\end{proof}

In addition to Lemma~{\ref{Reestimate}}, we need a further estimate of 
$\text{Re} \langle W_i-Z,  N(W_i) / ||N(W_i)|| \rangle$. 
For $\tau \geq 1$, let $E((m); (\alpha))_{\tau}$ be a domain defined by 
\begin{align}
&E((m); (\alpha))_{\tau}   \nn \\
&=\{
Z \in \mathbb{C}^{M_{s+1}} : 
||\tau z_{(1)}||^{2\alpha_1}+\dots+||\tau z_{(s)}||^{2\alpha_s}+
||\tau z_{(s+1)}||^{2}-1<0\}. \nn
\end{align}

\begin{lemma}\label{re1-t}
We fix $T \geq 1$. 
If $Z \in \overline{E((m); (\alpha))_{T}}$ and $W \in \pa E((m); (\alpha))$, then 
$\mathrm{Re} \langle W-Z, N(W)/||N(W)|| \rangle \geq 1-1/T$ holds.  
\end{lemma}

\begin{proof}
We fix $\tau$ with $T \leq \tau$. 
First, we find the minimum value of 
$\text{Re} \langle W-Z, N(W)/||N(W)|| \rangle$ for 
$Z \in \pa E((m); (\alpha))_{\tau}$ and $W \in \pa E((m); (\alpha))$,  
next we will show that its minimum value is $1-1/T$ when $\tau$ moves.  
$\text{Re} \langle W-Z, N(W)/||N(W)|| \rangle$ takes a minimal value when 
$\text{Re} \langle z_{(1)}, w_{(1)} \rangle, \dots, 
\text{Re} \langle z_{(s+1)}, w_{(s+1)} \rangle$ 
take maximal values. 
It is when the variable $Z \in \pa E((m); (\alpha))_{\tau}$ satisfies  
$z_{(1)}=A_{(1)}w_{(1)}, \dots, z_{(s+1)}=A_{(s+1)}w_{(s+1)}$ for some positive real numbers 
$A_{(1)}, \dots, A_{(s+1)}$. 
Therefore they satisfy the equality  
\begin{align}
&\left\{\left(A_{(1)}\tau\right)^{2\alpha_1}-\left(A_{(s+1)}\tau\right)^{2}\right\}||w_{(1)}||^{2\alpha_1}
+\dots+
\left\{\left(A_{(s)}\tau\right)^{2\alpha_s}-\left(A_{(s+1)}\tau\right)^{2}
\right\}||w_{(s)}||^{2\alpha_s}  \nn  \\
&-1+\left(A_{(s+1)}\tau \right)^2=0 \nn
\end{align}
and therefore, we conclude that 
$\text{Re} \langle W-Z, N(W)/||N(W)|| \rangle$ takes minimum value when 
$W$ and $Z$ satisfy $Z=W/\tau$. 
Under this relation, $\text{Re} \langle W-Z, N(W)/||N(W)|| \rangle$ is calculated as 
 \begin{align}
&\mathrm{Re} \langle W-Z, N(W)/||N(W)|| \rangle   \nn  \\ 
&=\dfrac{\; (1-\dfrac{1}{\; \tau \;})\left\{\alpha_{1}||w_{(1)}||^{2\alpha_{1}}+\dots
+\alpha_{s}||w_{(s)}||^{2\alpha_{s}}+||w_{(s+1)}||^{2}\right\} \;}
{\; \sqrt{\alpha_1^2||w_{(1)}||^{4\alpha_1-2}+\dots+\alpha_s^2||w_{(s)}||^{4\alpha_s-2}+
||w_{(s+1)}||^2} \;}.   \nn   
\end{align} 
We can show that 
\begin{equation}
\dfrac{\; (\alpha_{1}-1)||w_{(1)}||^{2\alpha_{1}}+\dots
+(\alpha_{s}-1)||w_{(s)}||^{2\alpha_{s}}+1 \;}
{\; \sqrt{\alpha_1^2||w_{(1)}||^{4\alpha_1-2}-||w_{(1)}||^{2\alpha_1}+\dots+
\alpha_s^2||w_{(s)}||^{4\alpha_s-2}-||w_{(s)}||^{2\alpha_s}+1} \;} \geq 1  \nn
\end{equation}
if $0 \leq ||w_{(1)}||^{2\alpha_1}+\dots+||w_{(s)}||^{2\alpha_s} \leq 1$ 
by induction.  
Hence, we obtain the estimate
$\text{Re} \langle W-Z, N(W)/||N(W)|| \rangle \geq 1-1/\tau$ 
and its minimum value is $1-1/T$  when $T \leq \tau$. 
\end{proof}

\section{construction of the entire mapping}

Given functions $f^1, \dots, f^{2N_{t+1}}, h$, we write 
$f_{(2k-1)}=(f^{2N_{k-1}+1}, \dots, f^{2N_{k-1}+n_k})$, 
$f_{(2k)}=(f^{2N_{k-1}+n_k+1}, \dots, f^{2N_{k}})$ 
and $F=(f_{(1)}, f_{(2)},  \dots, f_{(2t+1)}, f_{(2t+2)}, h)$.   
We use the similar notation for $g_{(1)}, \dots, g_{(2t+2)}$ and 
$G=(g_{(1)}, g_{(2)},  \dots, g_{(2t+1)}, g_{(2t+2)}, 0)$. 
Starting from $F$, we construct $g^{1}(Z), \dots, g^{2N_{t+1}}(Z)$ so that they satisfy 
certain estimates. 
In what follows, we read $\beta_{t+1}=1$. 
\begin{lemma}\label{mainlemma}
Assume that the following four conditions are satisfied, 

$\mathrm{(a)}$ $\varepsilon>0,  a>0$, 
$\varepsilon+\sum^{t+1}_{k=1}a^{2\beta_k}<1$ and 
$T \geq 1$, 

$\mathrm{(b)}$ $F=(f_{(1)}, \dots, f_{(2t+2)}, h)$ is a holomorphic mapping on $E((m);(\alpha))$ 
and it satisfies  
$|||F|||^{2\beta} <\sum^{t+1}_{k=1}a^{2\beta_k}$,  

$\mathrm{(c)}$ The each component $f^i$ satisfies $|f^i(Z)|<1$ and it extends 
continuously to $\overline{E((m);(\alpha))}$, 

$\mathrm{(d)}$ $||h(Z)||^2$ extends continuously to $\overline{E((m);(\alpha))}$. 

Then we can construct an entire mapping $G=(g_{(1)}, \dots, g_{(2t+2)}, 0)$ 
such that its component $g^i$ is of the form $\mathrm{(\ref{gi})}$ 
and satisfies the following three statements, 

$\mathrm{(i)}$ $|||F(Z)+G(Z)|||^{2\beta}<\varepsilon+\sum^{t+1}_{k=1}a^{2\beta_k}$
for each point $Z \in \pa E((m);(\alpha))$, 

$\mathrm{(ii)}$ $|g^{2N_{k-1}+1}(Z)|+\dots+|g^{2N_{k-1}+n_k}(Z)|<\varepsilon$ 
and 
$|g^{2N_{k-1}+n_k+1}(Z)|+\dots+|g^{2N_{k-1}+2n_k}(Z)|<\varepsilon$ 
for $Z \in \overline{E((m);(\alpha))_T}$, 

$\mathrm{(iii)}$ Let $Z \in \pa E((m);(\alpha))$. 
Assume that the following three inequalities 
\begin{align}
&||f_{(2k-1)}(Z)+g_{(2k-1)}(Z)||^2<
\dfrac{1}{\; 2 \;}\Bigl(\dfrac{\; a^2 \;}{\; 2^{1/\beta_k} \;}-\varepsilon^{1/(2t+2)} \Bigr),  \; 
k=1, \dots, t+1,  \nn \\
&||f_{(2k)}(Z)+g_{(2k)}(Z)||^2<
\dfrac{1}{\; 2 \;}\Bigl(\dfrac{\; a^2 \;}{\; 2^{1/\beta_k} \;}-\varepsilon^{1/(2t+2)}\Bigr), \; 
k=1, \dots, t,   \nn  \\
&||f_{(2t+2)}(Z)+g_{(2t+2)}(Z)||^2+||h(Z)||^2<
\dfrac{1}{\; 2 \;}\Bigl(\dfrac{\; a^2 \;}{2}-\varepsilon^{1/(2t+2)}\Bigr)   \nn
\end{align}
hold. 
Then the following two inequalities 
\begin{align}
&||f_{(2k-1)}(Z)+g_{(2k-1)}(Z)||^{2\beta_k}-||f_{(2k-1)}(Z)||^{2\beta_k}
>\Bigl(\dfrac{1}{\; 4 \;}\varepsilon^{1/(2t+2)} \Bigr)^{\beta_k}c,    \nn \\
&||f_{(2k)}(Z)+g_{(2k)}(Z)||^{2\beta_k}-||f_{(2k)}(Z)||^{2\beta_k}
>\Bigl(\dfrac{1}{\; 4 \;}\varepsilon^{1/(2t+2)} \Bigr)^{\beta_k}c  \nn
\end{align}
hold for $k=1, \dots, t+1$ and for some positive constant $c$. 
\end{lemma}

\begin{proof}
Let $\lambda$ and $r$ be as before. 
We define the sets of indices $i$ as  
\begin{align}
&I_{(2k-1)}(Z)=\{i :  2N_{k-1}+1 \leq i \leq 2N_{k-1}+n_k, Z \in B(W_i, \lambda r) \},   \nn   \\
&I_{(2k)}(Z)=\{i : 2N_{k-1}+n_k+1 \leq i \leq 2N_{k-1}+2n_k, Z  \in B(W_i, \lambda r) \}  \nn
\end{align}
for $k=1, \dots, t+1$. 
Note that 
\begin{equation}
I_{(2k)}(Z)=\{ i+n_k \; : \; i \in  I_{(2k-1)}(Z) \}.   \nn
\end{equation} 
Let $g^i$ be a function of the form (\ref{gi}), 
where $\gamma_{i}$ and $\gamma_{i+n_k}$ are determined by 
\begin{align}
&|\gamma_i|^2=\dfrac{1}{\; \sharp I_{(2k-1)}(Z) \;}
\left\{\dfrac{a^2}{\; 2^{1/\beta_k} \;}-||f_{(2k-1)}(W_i)||^2 \right\},      \nn \\
&|\gamma_{i+n_k}|^2=\dfrac{1}{\; \sharp I_{(2k)}(Z) \;} 
\left\{\dfrac{a^2}{\; 2^{1/\beta_k} \;}-||f_{(2k)}(W_i)||^2 \right\}  \nn  
\end{align}
for $k=1, \dots, t, i=2N_{k-1}+1, \dots, 2N_{k-1}+n_k$
and by 
\begin{align}
&|\gamma_i|^2=\dfrac{1}{\; \sharp I_{(2t+1)}(Z) \;}
\left\{\dfrac{\; a^2 \;}{2}-||f_{(2t+1)}(W_i)||^2 \right\},  \nn \\ 
&|\gamma_{i+n_{t+1}}|^2=\dfrac{1}{\; \sharp I_{(2t+2)}(Z) \;}
\left\{\dfrac{\; a^2 \;}{2}-||f_{(2t+2)}(W_i)||^2-||h(W_i)||^2 \right\} \nn  
\end{align}
for $i=2N_{t}+1, \dots, 2N_{t}+n_{t+1}$ and by 
\begin{align}
&f^i(W_i)\overline{\gamma_{i}}+\overline{f^{i}(W_i)}\gamma_{i}=0,   \nn \\ 
&f^{i+n_k}(W_i)\overline{\gamma_{i+n_k}}
+\overline{f^{i+n_k}(W_i)}\gamma_{i+n_k}=0.   \nn 
\end{align}
We gather the inequalities derived from the continuously extendability to 
$\overline{E((m);(\alpha))}$, which will be used later. 
For any $\eta$, we can find $r>0$ which satisfies Lemma~{\ref{gestimate}} 
such that if $Z \in B(W_i, \lambda r)$, 
then the following inequalities hold,   
\begin{align}
& \left| \; f^i(W_i)-f^i(Z) \; \right|<\eta,  \;  
\left| \; |f^i(W_i)|-|f^i(Z)| \; \right|<\eta,  \nn \\
& \left| \; ||f_{(2k-1)}(W_i)||^2-||f_{(2k-1)}(Z)||^2 \; \right|<\eta,  \; 
\left| \; ||f_{(2k)}(W_i)||^2-||f_{(2k)}(Z)||^2 \; \right|<\eta,    \nn  \\
&\left| \; ||h(W_i)||^2-||h(Z)||^2 \; \right|<\eta.   \nn
\end{align}

Proof (i). 
First, we calculate the case $k=1, \dots, t+1$. 
\begin{align}
&||f_{(2k-1)}(Z)+g_{(2k-1)}(Z)||^{2\beta_k}   \nn \\ 
&=
\Biggl[\sum_{\substack{i \not\in I_{(2k-1)}(Z)  \\ 2N_{k-1}+1 
\leq i \leq 2N_{k-1}+n_k}}   |f^{i}(Z)+g^{i}(Z)|^{2}
+\sum_{i \in I_{(2k-1)}(Z)} 
|f^{i}(Z)+g^{i}(Z)|^{2}
\Biggr]^{\beta_k}.   \nn
\end{align}
If $i \not\in I_{(2k-1)}(Z), 2N_{k-1}+1 \leq i \leq 2N_{k-1}+n_k$, then $|g^i(Z)|<\eta$, 
hence we have 
\begin{equation}
|f^i(Z)+g^i(Z)|^2<|f^i(Z)|^2+3\eta.  \nn
\end{equation}
On the other hand, if $i \in I_{(2k-1)}(Z)$, then, by continuity of 
$f^i(Z), |f^i(Z)|$ and $||f_{(2k-1)}(Z)||^2$ and by $|f^i(W_i)+g^i(Z)|<2$, 
we have 
\begin{align}
&|f^i(Z)+g^i(Z)|^2  \nn  \\
&<\{|f^i(W_i)-f^i(Z)|+|f^i(W_i)+g^i(Z)| \}^2  \nn \\
&<5\eta+|f^i(W_i)|^2+|\gamma_i|^2. \nn \\
&<\Bigl\{8+\dfrac{1}{\; \sharp I_{(2k-1)}(Z) \;} \Bigr\}\eta
+\dfrac{1}{\; \sharp I_{(2k-1)}(Z) \;}
\Bigl\{
\dfrac{a^2}{\; 2^{1/\beta_k} \;}-||f_{(2k-1)}(Z)||^2 \Bigr\}
+|f^i(Z)|^2.    \nn
\end{align}
Summing over indices $i=2N_{k-1}+1, \dots, 2N_{k-1}+n_k$, we obtain  
\begin{align}
&||f_{(2k-1)}(Z)+g_{(2k-1)}(Z)||^{2\beta_k}  \nn \\
&<
\Biggl[\sum_{\substack{i \not\in I_{(2k-1)}(Z)  \\ 2N_{k-1}+1 \leq i \leq 2N_{k-1}+n_k}}
\{|f^i(Z)|^2+3\eta\}+
\sum_{i \in I_{(2k-1)}(Z)} 
\Bigl\{
\bigl(8+\dfrac{1}{\; \sharp I_{(2k-1)}(Z) \;} \bigr)\eta  \nn \\
&\qquad \qquad 
+\dfrac{1}{\; \sharp I_{(2k-1)}(Z) \;}
\bigl(
\dfrac{a^2}{\; 2^{1/\beta_k} \;}-||f_{(2k-1)}(Z)||^2 \bigr)
+|f^i(Z)|^2
\Bigr\}
\Biggr]^{\beta_k}   \nn \\
&=
\Bigl[\bigl\{1+\sum_{\substack{i \not\in I_{(2k-1)}(Z)  \\ 2N_{k-1}+1 \leq i \leq 2N_{k-1}+n_k}}3+
\sum_{i \in I_{(2k-1)}(Z)}8 \bigr\}\eta
+\dfrac{a^2}{\; 2^{1/\beta_k} \;} \Bigr]^{\beta_k}.  \nn
\end{align}
We can calculate 
$||f_{(2k)}(Z)+g_{(2k)}(Z)||^{2\beta_k}$ in the same way.  
The calculation for $||f_{(2t+2)}(Z)+g_{(2t+2)}(Z)||^2+||h(Z)||^2$ also goes in the same way except for 
the definition of $\gamma_i$ and the result is 
\begin{align}
&||f_{(2t+2)}(Z)+g_{(2t+2)}(Z)||^2+||h(Z)||^2  \nn \\
&<
\bigl\{2+\sum_{\substack{i \not\in I_{(2t+2)}(Z)  \\ 2N_{t}+n_{t+1}+1 \leq i \leq 2N_{t}+2n_{t+1}}}3+
\sum_{i \in I_{(2t+2)}(Z)}8 \bigr\}\eta
+\dfrac{\; a^2 \;}{2}. \nn 
\end{align}
Summarizing above results, we obtain 
\begin{align}
&|||F(Z)+G(Z)|||^{2\beta}    \nn \\
&<\sum^{t+1}_{k=1}
\Biggl[
\Big\{\bigl(1+\sum_{\substack{i \not\in I_{(2k-1)}(Z)  \\ 2N_{k-1}+1 \leq i \leq 2N_{k-1}+n_k}}3+
\sum_{i \in I_{(2k-1)}(Z)}8 \bigr)\eta
+\dfrac{a^2}{\; 2^{1/\beta_k} \;} \Bigr\}^{\beta_k}   \nn \\
&\qquad \qquad 
+\Bigl\{\bigl(1+\sum_{\substack{i \not\in I_{(2k)}(Z)  \\ 2N_{k-1}+n_k+1 \leq i \leq 2N_{k-1}+2n_k}}3+
\sum_{i \in I_{(2k)}(Z)}8 \bigr)\eta
+\dfrac{a^2}{\; 2^{1/\beta_k} \;} \Bigr\}^{\beta_k}
\Biggr]
+\eta   \nn\\
&=
\sum^{t+1}_{k=1}a^{2\beta_k}  \nn \\
&\quad 
+\sum^{t+1}_{k=1}
\Biggl[
\sum^{\beta_k}_{j=1}
\begin{pmatrix}
\beta_k \\
j
\end{pmatrix}
\eta^j
\Bigl(\dfrac{a^2}{\; 2^{1/\beta_k} \;} \Bigr)^{\beta_k-j}
\Big\{\bigl(
1+\sum_{\substack{i \not\in I_{(2k-1)}(Z)  \\ 2N_{k-1}+1 \leq i \leq 2N_{k-1}+n_k}}3+
\sum_{i \in I_{(2k-1)}(Z)}8 \bigr)^j    \nn \\
&\qquad \qquad \qquad 
+\bigl(1+\sum_{\substack{i \not\in I_{(2k)}(Z)  \\ 2N_{k-1}+n_k+1 \leq i \leq 2N_{k-1}+2n_k}}3+
\sum_{i \in I_{(2k)}(Z)}8
\bigr)^j  \Bigr\}
\Biggr]+\eta. \nn
\end{align}
For sufficiently small $\varepsilon>0$, we can choose $\eta$ such that 
\begin{align}
&\sum^{t+1}_{k=1}
\Biggl[
\sum^{\beta_k}_{j=1}
\begin{pmatrix}
\beta_k \\
j
\end{pmatrix}
\eta^j
\Bigl(\dfrac{a^2}{\; 2^{1/\beta_k} \;} \Bigr)^{\beta_k-j}
\Bigl[\bigl\{
1+\sum_{\substack{i \not\in I_{(2k-1)}(Z)  \\ 2N_{k-1}+1 \leq i \leq 2N_{k-1}+n_k}}3+
\sum_{i \in I_{(2k-1)}(Z)}8 \bigr\}^j   \label{etavarepsilon} \\
&\qquad \qquad \qquad 
+\bigl\{1+\sum_{\substack{i \not\in I_{(2k)}(Z)  \\ 2N_{k-1}+n_k+1 \leq i \leq 2N_{k-1}+2n_k}}3+
\sum_{i \in I_{(2k)}(Z)}8
\bigr\}^j  \Bigr]
\Biggr]+\eta
<\varepsilon.    \nn
\end{align}
Note that $O(\varepsilon)=O(\eta)$. We shall use this fact later. 
This proves (i). 

Proof (ii). 
By Lemma~{\ref{re1-t}} with $W=W_i$, 
we obtain the inequality, 
 \begin{align} 
&|g^{2N_{k-1}+1}(Z)|+\dots+|g^{2N_{k-1}+n_k}(Z)|     \nn \\
&<
e^{-\nu\text{Re} \langle W_{2N_{k-1}+1}-Z, N(W_{2N_{k-1}+1})/||N(W_{2N_{k-1}+1})|| \rangle}  \nn \\
&\qquad +\dots+
e^{-\nu\text{Re} \langle W_{2N_{k-1}+n_k}-Z, N(W_{2N_{k-1}+n_k})/||N(W_{2N_{k-1}+n_k})|| \rangle}  \nn \\
&<
n_ke^{-\nu(1-1/T)}. \nn
\end{align}
For any $\varepsilon>0$, we can find $\nu>0$ in Lemma~{\ref{gestimate}} such that 
$n_ke^{-\nu(1-1/T)}<\varepsilon$ holds.  
The second estimate is proved by the same manner. 
This proves (ii). 

Proof (iii). 
For $k=1, \dots, t+1$, we let 
\begin{equation} 
D_{2k-1}=||f_{(2k-1)}(Z)+g_{(2k-1)}(Z)||^{2\beta_k}-||f_{(2k-1)}(Z)||^{2\beta_k}    \nn
\end{equation}
and claim that it satisfies 
\begin{equation}
D_{2k-1}>\Bigl(\dfrac{1}{\; 4 \;}\varepsilon^{1/(2t+2)} \Bigr)^{\beta_k}c \label{d2k-1estimate}
\end{equation}
for some positive constant $c$. 

If $i \not \in I_{(2k-1)}(Z)$, then 
\begin{equation}
|f^i(Z)+g^i(Z)|^2=|f^i(Z)|^2+O(\eta).  \label{d2k-11}
\end{equation}
If $i \in I_{(2k-1)}(Z)$, then, by continuity of $f^i(Z)$, 
we have 
\begin{align}
|f^i(Z)+g^i(Z)|^2=&|f^i(Z)-f^i(W_i)+f^i(W_i)+\gamma_i\phi_i(Z)|^2   \label{d2k-12}  \\
=&|f^i(W_i)|^2+|\gamma_i|^2|\phi_i(Z)|^2+O(\eta)   \nn 
\end{align}
and 
\begin{equation}
|f^i(Z)|^2=|f^i(Z)-f^i(W_i)+f^i(W_i)|^2=
|f^i(W_i)|^2+O(\eta). \label{d2k-13}
\end{equation}
Therefore 
(\ref{d2k-11}), (\ref{d2k-12}) and (\ref{d2k-13}) imply that 
\begin{align}
&D_{2k-1}  \label{d2k-1}\\
&=\Bigl\{
\sum_{\substack{i \not\in I_{(2k-1)}(Z)  \\ 2N_{k-1}+1 \leq i \leq 2N_{k-1}+n_k}}|f^i(Z)|^2 \nn \\
&\qquad \qquad +
\sum_{i \in I_{(2k-1)}(Z)}
\bigl(|f^i(W_i)|^2+|\gamma_i|^2|\phi_i(Z)|^2 \bigr)+O(\eta)\Bigr\}^{\beta_k}  \nn  \\
&\qquad 
-\Bigl\{
\sum_{\substack{i \not\in I_{(2k-1)}(Z)  \\ 2N_{k-1}+1 \leq i \leq 2N_{k-1}+n_k}}|f^i(Z)|^2+
\sum_{i \in I_{(2k-1)}(Z)}
|f^i(W_i)|^2+O(\eta)\Bigr\}^{\beta_k}  \nn  \\
&>O(\eta).  \nn
\end{align}
Suppose that 
$||f_{(2k-1)}(Z)+g_{(2k-1)}(Z)||^2<
\dfrac{1}{\; 2 \;}\Bigl(a^2/2^{1/\beta_k}-\varepsilon^{1/(2t+2)}\Bigr)$
for some $Z \in \pa E((m);(\alpha))$. 
Choose a ball $B(W_i, \lambda r)$ containing $Z$. 
Then the continuity of $||f_{(2k-1)}(Z)||^2$ and inequality in (\ref{d2k-1}) 
imply that 
\begin{align}
||f_{(2k-1)}(W_i)||^2&<\eta+||f_{(2k-1)}(Z)||^2  \nn \\
&<||f_{(2k-1)}(Z)+g_{(2k-1)}(Z)||^2+O(\eta)  \nn \\
&<\dfrac{1}{\; 2 \;}\Bigl(\dfrac{a^2}{\; 2^{1/\beta_k} \;}-\varepsilon^{1/(2t+2)}\Bigr)+O(\eta). \nn
\end{align}
Since $O(\eta)=O(\varepsilon)$, 
we have 
\begin{equation}
||f_{(2k-1)}(W_i)||^2<\dfrac{a^2}{\; 2^{1/\beta_k} \;}-\dfrac{1}{\; 2 \;}\varepsilon^{1/(2t+2)}. \nn
\end{equation}
This estimate implies 
\begin{align} 
\sum_{i \in I_{(2k-1)}(Z)}|\gamma_i|^2
>\dfrac{1}{\; 2 \;}\varepsilon^{1/(2t+2)}.    \nn 
\end{align}
Since the boundary $\pa E((m);(\alpha))$ is compact, 
there exists $r_0>0$ such that, for any $i$ and $Z \in \pa E((m);(\alpha))$, 
$||W_i-Z||<r_0$ holds.  
Therefore, in view of (\ref{phiest}), 
we have 
$|\phi_i(Z)|>e^{-\nu A||W_i-Z||^{2p}}>e^{-\nu Ar^{2p}_{0}}$. 
Substitute this into the equality in (\ref{d2k-1}) to obtain (\ref{d2k-1estimate}) as 
\begin{equation}
D_{2k-1}>\Bigl(
\dfrac{1}{\; 2 \;}\varepsilon^{1/(2t+2)}e^{-2\nu Ar^{2p}_{0}}
\Bigr)^{\beta_k}+O(\varepsilon)>
\Bigl(\dfrac{1}{\;4 \;}\varepsilon^{1/(2t+2)} \Bigr)^{\beta_k}c \nn
\end{equation}
for some positive constant $c$. 
Applying the same argument to  
$D_{2k}$ for $k=1, \dots, t+1$,
we obtain the same estimates as (\ref{d2k-1estimate}), which are the conclusions of (iii). 
\end{proof}

Let $h=(h_1, \dots, h_p)$ be a holomorphic mapping on $E((m);(\alpha))$ such that 
$||h||^2$ extends continuously to $\overline{E((m);(\alpha))}$. 
Assume $\text{max}\{\beta_1, \dots, \beta_{t}\}<t+1$. 

We choose two sequences $\{a_{\ell}\}$ and $\{\varepsilon_{\ell}\}$ 
satisfying the following conditions.

$\mathrm{(i)}$  
$\{a_{\ell}\}$ is a positive increasing sequence and 
$\{\varepsilon_{\ell}\}$ is a positive decreasing sequence. 

$\mathrm{(ii)}$  
$\{\sum^{t+1}_{k=1}a^{2\beta_k}_{\ell}\}$ is strictly increasing to $1$, 

$\mathrm{(iii)}$
$\{\varepsilon_{\ell}\}$ decreases to $0$,  
$\sum_{\ell} \varepsilon^{1/2}_{\ell}$ converges and 
$\sum_{\ell}\varepsilon^{\text{max}\{\beta_k\}/(2t+2)}_{\ell}$ diverges. 

$\mathrm{(iv)}$
$\sum^{t+1}_{k=1}a^{2\beta_k}_{\ell-1}+\varepsilon_{\ell}<\sum^{t+1}_{k=1}a^{2\beta_k}_{\ell}
\; \text{and} \; 
\text{sup}_{\pa E((m);(\alpha))}||h||^2<\sum^{t+1}_{k=1}a^{2\beta_k}_{0}$. 

Put $F_0(Z)=(0, \dots, 0, h(Z))$. 
Using Lemma~{\ref{mainlemma}}, we construct a sequence of entire mappings $\{G_{\ell}\}$ 
and a positive decreasing sequence $\{T_{\ell}\}$ inductively starting from $F_0$. 
Put $F_{\ell}=F_0+\sum^{\ell}_{j=1}G_j$. 

\begin{proposition}\label{inductiveproposition} 
Take $\{a_{\ell}\}, \{\varepsilon_{\ell}\}$ and $h$ as above. 
Then we can construct a sequence $\{G_{\ell}\}$ of entire mappings of the form 
\begin{equation}
G_{\ell}=(g_{(1)\ell}, g_{(2)\ell}, \dots, g_{(2t+1)\ell}, g_{(2t+2)\ell}, 0) : 
\mathbb{C}^{M_{s+1}} \to \mathbb{C}^{2N_{t+1}+p},    \nn
\end{equation}
where the components of $G_{\ell}$ are of the form $\mathrm(\ref{gi})$ and  
a decreasing sequence $\{T_{\ell}\}$ with $\lim_{\ell \to \infty}T_{\ell}=1$ and 
$\bigcup_{\ell}E((m);(\alpha))_{T_{\ell}}=E((m);(\alpha))$ 
inductivety satisfying the following. 

$\mathrm{(i)}$  
$|||F_{\ell-1}(Z)|||^{2\beta} \geq \mathrm{min}_{W \in \pa E((m);(\alpha))}
|||F_{\ell-1}(W)|||^{2\beta}-2^{-\ell}$ \\
for $Z \in \overline{E((m);(\alpha))} \backslash \overline{E((m);(\alpha))_{T_{\ell}}}$. 

$\mathrm{(ii)}$ 
$|||F_{\ell}(Z)|||^{2\beta}<\sum^{t+1}_{k=1}a^{2\beta_k}_{\ell}$ for 
$Z \in \overline{E((m);(\alpha))}$. 

$\mathrm{(iii)}$ 
\begin{equation}
\sum^{2N_{k-1}+n_k}_{i=2N_{k-1}+1}|g^i_{\ell}(Z)|<\varepsilon_{\ell}, \quad 
\sum^{2N_{k-1}+2n_k}_{i=2N_{k-1}+n_k+1}|g^i_{\ell}(Z)|<\varepsilon_{\ell}, \nn
\end{equation}
where 
$g_{(2k-1) \ell}=(g^{2N_{k-1}+1}_{\ell}, \dots, g^{2N_{k-1}+n_k}_{\ell})$
and 
$g_{(2k) \ell}=(g^{2N_{k-1}+n_k+1}_{\ell}, \dots, g^{2N_{k-1}+2n_k}_{\ell})$
for $Z \in \overline{E((m);(\alpha))_{T_{\ell}}}$. 

$\mathrm{(iv)}$ 
Let $Z \in \pa E((m);(\alpha))$. 
Assume that the following three inequalities  
\begin{align}
& ||f_{(2k-1)\ell}(Z)||^2<
\dfrac{1}{\; 2 \;}\Bigl(\dfrac{\; a^2_{\ell-1} \;}{\; 2^{1/\beta_k} \;}-\varepsilon^{1/(2t+2)}_{\ell}\Bigr), 
\; k=1, \dots, t+1, \nn \\
&||f_{(2k)\ell}(Z)||^2<
\dfrac{1}{\; 2 \;}\Bigl(\dfrac{\; a^2_{\ell-1} \;}{\; 2^{1/\beta_k} \;}-\varepsilon^{1/(2t+2)}_{\ell} \Bigr), 
\; k=1, \dots, t,   \nn \\
&||f_{(2t+2)\ell}(Z)||^2+||h(Z)||^2<
\dfrac{1}{\; 2 \;}\Bigl(\dfrac{\; a^2_{\ell-1} \;}{2}-\varepsilon^{1/(2t+2)}_{\ell}\Bigr)  \nn
\end{align}
hold. 
Then two estimates 
\begin{align}
&||f_{(2k-1)\ell}(Z)||^{2\beta_k}-||f_{(2k-1)\ell-1}(Z)||^{2\beta_k}
>\Bigl(\dfrac{1}{\; 4 \;}\varepsilon^{1/(2t+2)}_{\ell}\Bigr)^{\beta_k}c,   \nn \\
&||f_{(2k)\ell}(Z)||^{2\beta_k}-||f_{(2k)\ell-1}(Z)||^{2\beta_k}
>\Bigl(\dfrac{1}{\; 4 \;}\varepsilon^{1/(2t+2)}_{\ell}\Bigr)^{\beta_k}c   \nn
\end{align}
hold for $k=1, \dots, t+1$ and for some positive constant $c$. 
\end{proposition}

\begin{proof}
The case of $\ell=1$. 
First, we construct $G_1$. 
Since $|||F_0(Z)|||^{2\beta}=||h||^2$ is continuous on $\overline{E((m);(\alpha))}$, 
there exists $T_0>0$ satisfying (i). 
We apply Lemma~{\ref{mainlemma}} to the data 
$a=a_0, \varepsilon=\varepsilon_1, T=T_0$ and $F=F_0$ to find  
$G=G_1$. These data satisfy the assumptions of Lemma~{\ref{mainlemma}}. 
Therefore we can construct $G=G_1$ with the properties (i), (ii), (iii) 
of Lemma~{\ref{mainlemma}}, which implies the 
properties (ii), (iii) and (iv) of this proposition for $\ell=1$. 
Next the case of $\ell=2$. 
Now we have $F_1=F_0+G_1$. 
By the continuity of $F_1$ on $\overline{E((m);(\alpha))}$, we can find $T_1$ 
such that (i) holds with $\ell=2$.   
We apply Lemma~{\ref{mainlemma}} to the data $a=a_1, \varepsilon=\varepsilon_2, T=T_1$ 
and $F=F_1$ to find 
$G=G_2$ by the similar argument as $\ell=1$. 
We can proceed this process to find the sequences $\{G_{\ell}\}$ and $\{T_{\ell}\}$ with the desired 
properties. 
\end{proof}

\section{Construction of proper holomorphic mapping}

Property (iii) in Proposition~\ref{inductiveproposition} implies that 
$F=\lim F_\ell=F_0+\sum^{\infty}_{j=1}G_j$ converges uniformly on a compact subset of 
$E((m);(\alpha))$ and hence it defines a holomorphic mapping there.   
We will prove that the mapping $F$ is a proper holomorphic mapping 
between $E((m);(\alpha))$ and $E((n);(\beta))$. 
To this end, we need one lemma. 

\begin{lemma}\label{fgh}
Let $F=(f_1, \dots, f_n), G=(g_1, \dots, g_n)$ and $H=(h_1, \dots, h_n)$ be $n$-dimensional vectors. 
For $\mu=(\mu_1, \mu_2, \mu_3) \in \mathbb{Z}^{3}_{\geq 0}$, we put 
$\mu!=\mu_1!\mu_2!\mu_3!$ and $|\mu|=\mu_1+\mu_2+\mu_3$. 
Then the estimate 
\begin{equation}
||F+G+H||^{2\alpha} 
\leq 
||F||^{2\alpha}  
+\sum_{\substack{|\mu|=2\alpha   \\\mu_1\not=2\alpha}}
\dfrac{(2\alpha)!}{\mu!}
||G||^{\mu_2}||H||^{\mu_3}   \nn
\end{equation}
holds for $\alpha \in \mathbb{N}$. 
\end{lemma}

\begin{theorem}\label{proper}
The limit $F=\lim F_{\ell}$ is a proper holomorphic mapping between $E((m);(\alpha))$ and $E((n);(\beta))$.  
\end{theorem}

\begin{proof}
Proposition~{\ref{inductiveproposition}} (ii) implies that 
$|||F|||^{2\beta}\leq 1$ on $\overline{E((m);(\alpha))}$. 
By the maximal modulus principle, $|||F|||^{2\beta}<1$ on $E((m);(\alpha))$. 
Hence, $F$ is a mapping between $E((m);(\alpha))$ and $E((n);(\beta))$. 
Construct $\{T_{\ell}\}$ according to Proposition~{\ref{inductiveproposition}} (i). 

For $Z \in \pa E((m);(\alpha))$, define  
\begin{equation}
u_{(2k-1)\ell}(Z)=\mathrm{min}
\Bigl\{||f_{(2k-1)\ell}(Z)||^{2\beta_k},  \; 
\dfrac{1}{\; 2 \;}\Bigl(\dfrac{\; a^2_{\ell-1} \;}{\; 2^{1/\beta_k} \;}
-\varepsilon^{1/(2t+2)}_{\ell} \Bigr)^{\beta_k} \Bigr\}.   \nn
\end{equation}
We claim that $\{u_{(2k-1)\ell}(Z)\}$ is an increasing sequence. 
If 
$||f_{(2k-1)\ell}(Z)||^{2\beta_k} \geq 
\dfrac{1}{\; 2 \;}\Bigl(\dfrac{\; a^2_{\ell-1} \;}{\; 2^{1/\beta_k} \;}
-\varepsilon^{1/(2t+2)}_{\ell} \Bigr)^{\beta_k}$, 
then 
$u_{(2k-1)\ell}(Z)=\dfrac{1}{\; 2 \;}\Bigl(\dfrac{\; a^2_{\ell-1} \;}{\; 2^{1/\beta_k} \;}
-\varepsilon^{1/(2t+2)}_{\ell} \Bigr)^{\beta_k}$
and this defines an increasing sequence. 
If 
$||f_{(2k-1)\ell}(Z)||^{2\beta_k}<
\dfrac{1}{\; 2 \;}\Bigl(\dfrac{\; a^2_{\ell-1} \;}{\; 2^{1/\beta_k} \;}
-\varepsilon^{1/(2t+2)}_{\ell} \Bigr)^{\beta_k}$, 
then 
Proposition~\ref{inductiveproposition} (iv) implies that 
\begin{align}
u_{(2k-1)\ell}(Z) &=||f_{(2k-1)\ell}(Z)||^{2\beta_k} > ||f_{(2k-1)\ell-1}(Z)||^{2\beta_k}
+ \Bigl(\dfrac{1}{\; 4 \;}\varepsilon^{1/(2t+2)}_{\ell} \Bigr)^{\beta_k}c  \label{uincrease} \\
&\geq u_{(2k-1)\ell-1}(Z)+ \Bigl(\dfrac{1}{\; 4 \;}\varepsilon^{1/(2t+2)}_{\ell} \Bigr)^{\beta_k}c   \nn \\
&>u_{(2k-1)\ell-1}(Z).  \nn
\end{align}
Hence, in both cases, $\{u_{(2k-1)\ell}(Z) \}$ is an increasing sequence. 
The same argument shows that 
$\{u_{(2k)\ell}(Z) \}, k=1, \dots, t+1$, is an increasing sequence. 

Define 
\begin{equation}
u_{\ell}(Z)=\sum^{t+1}_{k=1}\{u_{(2k-1)\ell}(Z)+u_{(2k)\ell}(Z)\},   \nn
\end{equation}
then we can show that $\lim _{\ell \to \infty}u_{\ell}(Z)=1$ 
uniformly on $\pa E((m);(\alpha))$. 
Suppose that $\lim _{\ell \to \infty}u_{\ell}(Z)<1$ for some $Z \in \pa E((m);(\alpha))$. 
Since we have 
\begin{equation}
\lim_{\ell \to \infty}\sum^{t+1}_{k=1}
2\Bigl(\dfrac{\; a^2_{\ell-1} \;}{\; 2^{1/\beta_k} \;}-\varepsilon^{1/(2t+2)}_{\ell} \Bigr)^{\beta_k}=1,   \nn
\end{equation}
we can find a sufficiently large integer $\ell_0$ such that, 
for each $\ell>\ell_0$, there exists $k$ such that 
$u_{(2k-1)\ell}(Z)=||f_{(2k-1)\ell}(Z)||^{2\beta_k}$ or   
$u_{(2k)\ell}(Z)=||f_{(2k)\ell}(Z)||^{2\beta_k}$, namely, 
$||f_{(2k-1)\ell}(Z)||^{2\beta_k}<
\dfrac{1}{\; 2 \;}\Bigl(\dfrac{\; a^2_{\ell-1} \;}{\; 2^{1/\beta_k} \;}
-\varepsilon^{1/(2t+2)}_{\ell} \Bigr)^{\beta_k}$ 
or
$||f_{(2k)\ell}(Z)||^{2\beta_k}<
\dfrac{1}{\; 2 \;}\Bigl(\dfrac{\; a^2_{\ell-1} \;}{\; 2^{1/\beta_k} \;}
-\varepsilon^{1/(2t+2)}_{\ell} \Bigr)^{\beta_k}$ holds.  
Assume that the former inequality holds.  
There exists at least one $k$, say $k_0$, which appears infinitely many times. 
Let $L=\{\ell_1, \ell_2, \ell_3, \cdots; \ell_1<\ell_2<\ell_3< \cdots\}$ be a set of $\ell$ satisfying 
$||f_{(2k_0-1)\ell}(Z)||^{2\beta_{k_0}}<
\dfrac{1}{\; 2 \;}\Bigl(\dfrac{\; a^2_{\ell-1} \;}{\; 2^{1/\beta_{k_0}} \;}
-\varepsilon^{1/(2t+2)}_{\ell} \Bigr)^{\beta_{k_0}}$.
Then 
\begin{align}
u_{(2k_0-1)\ell}(Z)
\geq u_{(2k_0-1)\ell-1}(Z)+ \Bigl(\dfrac{1}{\; 4 \;}\varepsilon^{1/(2t+2)}_{\ell} \Bigr)^{\beta_{k_0}}c   \nn
\end{align}
holds by (\ref{uincrease}). 
Summing over indices $\ell_1, \dots, \ell_{\lambda} \in L$, then 
we get 
\begin{align}
&\sum^{\lambda-1}_{j=1}
\bigl\{u_{(2k_0-1)\ell_j}(Z)-u_{(2k_0-1)\ell_{j+1}-1}(Z)\bigr\}
+u_{(2k_0-1)\ell_{\lambda}}(Z)   \nn \\
&\geq  
u_{(2k_0-1)\ell_1-1}(Z)
+\sum^{\ell_{\lambda}}_{j=\ell_1}
\Bigl(\dfrac{1}{\; 4 \;}\varepsilon^{1/(2t+2)}_{\ell_j} \Bigr)^{\beta_{k_0}}c.  \nn
\end{align}
Since $\{u_{(2k_0-1)\ell}(Z)\}$ is a strictly increasing sequence, 
each term in the summation in the left hand side is nonpositive. 
Hence 
we obtain 
\begin{equation}
1>u_{(2k_0-1)\ell_{\lambda}}(Z) 
\geq  
u_{(2k_0-1)\ell_1-1}(Z)
+\sum^{\ell_{\lambda}}_{j=\ell_1}
\Bigl(\dfrac{1}{\; 4 \;}\varepsilon^{1/(2t+2)}_{\ell_j} \Bigr)^{\beta_{k_0}}c.   \nn
\end{equation}
Since the summation 
$\sum_{\ell}\varepsilon^{\text{max}\{\beta_k\}/(2t+2)}_{\ell}$ diverges, 
 the summation in the right hand side also diverges when $\lambda$ goes to infinity. 
This is a contradiction. 
Now we have proved 
$\lim u_{\ell}(Z)=1$ and this means that 
$\lim |||F_{\ell}(Z)|||^{2\beta} \geq \lim u_{\ell}(Z)=1$. 
On the other hand, since $F=\lim F_{\ell}$ is a mapping to $E((n);(\beta))$, we have 
$\lim |||F_{\ell}(Z)|||^{2\beta} \leq 1$. 
Hence, we obtain $\lim |||F_{\ell}(Z)|||^{2\beta}=1$ 
uniformly on $\pa E((m);(\alpha))$. 

Let $Z \in \overline{E_{T_{\ell+1}}} \backslash \overline{E_{T_{\ell}}}$. 
We put 
\begin{align}
& F_{\ell-1}(Z)=(f_{(1)\ell-1}(Z), \dots, f_{(2t+2)\ell-1}(Z), h(Z)), \nn  \\
&G_{\ell}(Z)=(g_{(1)\ell}(Z), \dots, g_{(2t+2)\ell}(Z), 0), \nn \\
&\sum^{\infty}_{j=\ell+1}G_{j}(Z)=(g^{\infty}_{(1)\ell+1}(Z), \dots, g^{\infty}_{(2t+2)\ell+1}(Z), 0),  \nn \\
&F(Z)=(f_{(1)}(Z), \dots, f_{(2t+2)}(Z), h(Z)).  \nn
\end{align}
Applying Lemma~{\ref{fgh}} to each component of $F_{\ell-1}=F-G_{\ell}-\sum^{\infty}_{j=\ell+1}G_{j}$, 
we obtain, for $k=1, \dots, t+1$, 
\begin{align}
&||f_{(2k-1)\ell-1}(Z)||^{2\beta_k} \label{f2k-1infty}  \\
& \quad \leq 
||f_{(2k-1)}(Z)||^{2\beta_k} 
+\sum_{\substack{|\mu|=2\beta_{k} \\ \mu_1\not=2\beta_k}}
\dfrac{\; (2\beta_k)! \;}{\mu \; !}
||-g_{(2k-1)\ell}||^{\mu_2}||-g^{\infty}_{(2k-1)\ell+1}||^{\mu_3},    \nn 
\end{align}
and, for $k=1, \dots, t$, 
\begin{align}
&||f_{(2k)\ell-1}(Z)||^{2\beta_k}  \label{f2kinfty}\\
& \quad \leq 
||f_{(2k)}(Z)||^{2\beta_k} 
+\sum_{\substack{|\mu|=2\beta_{k} \\ \mu_1\not=2\beta_k}}
\dfrac{\; (2\beta_k)! \;}{\mu \; !}
||-g_{(2k)\ell}||^{\mu_2}||-g^{\infty}_{(2k)\ell+1}||^{\mu_3}   \nn 
\end{align}
and 
\begin{align}
& ||f_{(2t+2)\ell-1}(Z)||^{2}+||h(Z)||^2 \label{f2t+2infty}\\
& \quad \leq 
||f_{(2t+2)}(Z)||^{2}+||h(Z)||^2 
+\sum_{\substack{|\mu|=2 \\ \mu_1\not=2}}
\dfrac{2}{\; \mu \; ! \;}
||-g_{(2t+2)\ell}||^{\mu_2}||-g^{\infty}_{(2t+2)\ell+1}||^{\mu_3}.   \nn 
\end{align}
Summing (\ref{f2k-1infty}), (\ref{f2kinfty}) and (\ref{f2t+2infty}), 
we obtain 
\begin{align}
&\mathrm{min}_{W \in \pa E((m);(\alpha))}
|||F_{\ell-1}(W)|||^{2\beta}-2^{-\ell} \label{ellinfty}\\
&\leq |||F_{\ell-1}(Z)|||^{2\beta}  \nn \\
&\leq 
|||F(Z)|||^{2\beta}  \nn \\
&\qquad 
+\sum^{t+1}_{k=1}\sum_{\substack{|\mu|=2\beta_{k} \\ \mu_1\not=2\beta_k}}
\dfrac{\; (2\beta_k)! \;}{\; \mu \; ! \;}
\Bigl\{
||-g_{(2k-1)\ell}(Z)||^{\mu_2}||-g^{\infty}_{(2k-1)\ell+1}(Z)||^{\mu_3}   \nn \\
&\qquad \qquad \qquad \qquad \qquad 
+||-g_{(2k)\ell}(Z)||^{\mu_2}||-g^{\infty}_{(2k)\ell+1}(Z)||^{\mu_3}
\Bigr\}.   \nn
\end{align}
The first inequality is a property of Proposition~\ref{inductiveproposition} (i). 
In order to prove the summation-term in (\ref{ellinfty}) goes to zero 
when $\ell$ goes to infinity, 
we divide the summation into the cases $\mu_2=0$ and $\mu_2\not=0$. 

(I) $\mu_2=0$. 
Since we have the following inequalities 
\begin{align}
&||-g^{\infty}_{(2k-1)\ell+1}(Z)||^{\mu_3}    \nn \\
&=\Bigl[
\bigl|\sum^{\infty}_{j=\ell+1}g^{2N_{k-1}+1}_{j}(Z)\bigr|^2+\dots+
\bigl|\sum^{\infty}_{j=\ell+1}g^{2N_{k-1}+n_k}_{j}(Z)\bigr|^2
\Bigr]^{\mu_3/2}  \nn \\
&<
\Bigl[
\bigl\{\sum^{\infty}_{j=\ell+1}|g^{2N_{k-1}+1}_{j}(Z)| \bigr\}^2+\dots+
\bigl\{\sum^{\infty}_{j=\ell+1}|g^{2N_{k-1}+n_k}_{j}(Z)| \bigr\}^2
\Bigr]^{\mu_3/2}  \nn \\
&<\Bigl\{
\sum^{\infty}_{j=\ell+1}\sum^{2N_{k-1}+n_k}_{i=2N_{k-1}+1}
|g^{i}_{j}(Z)|
\Bigr\}^{\mu_3}  \nn  \\
&<\bigl\{ \sum^{\infty}_{j=\ell+1}\varepsilon_j 
\bigr\}^{\mu_3}, \nn
\end{align}
the summation-term in (\ref{ellinfty}) with $\mu_2=0$ is estimated as 
\begin{align}
&\sum^{t+1}_{k=1}\sum_{\substack{|\mu|=2\beta_{k} \\ \mu_1\not=2\beta_k  \\  \mu_2=0}}
\dfrac{\; (2\beta_k)! \;}{\; \mu \; ! \;}
\Bigl\{||-g^{\infty}_{(2k-1)\ell+1}(Z)||^{\mu_3} 
+||-g^{\infty}_{(2k)\ell+1}(Z)||^{\mu_3}
\Bigr\}  \nn  \\
&<\sum^{t+1}_{k=1}\sum_{\substack{|\mu|=2\beta_{k} \\ \mu_1\not=2\beta_k  \\  \mu_2=0}}
\dfrac{\; (2\beta_k)! \;}{\; \mu \; ! \;}
\bigl\{ 2\sum^{\infty}_{j=\ell+1}\varepsilon_j 
\bigr\}^{\mu_3}. \nn
\end{align}
Since $\sum \varepsilon^{1/2}_j$ converges, 
this goes to zero when $\ell$ goes to infinity. 

(II)$\mu_2\not=0$. 
By Proposition~{\ref{inductiveproposition}} (iii), 
if $\ell$ is sufficinetly large, 
we have the inequalities 
$||-g_{(2k-1)\ell}(Z)||^{\mu_2}<\varepsilon^{\mu_2}_{\ell}$ 
and 
$||-g_{(2k)\ell}(Z)||^{\mu_2}<\varepsilon^{\mu_2}_{\ell}$. 
Since 
$||-g^{\infty}_{(2k)\ell+1}(Z)||^{\mu_3} \leq 1$ and $||-g^{\infty}_{(2k-1)\ell+1}(Z)||^{\mu_3} \leq 1$, 
we obtain 
\begin{align}
&\sum^{t+1}_{k=1}\sum_{\substack{|\mu|=2\beta_{k} \\ \mu_1\not=2\beta_k}}
\dfrac{\; (2\beta_k)! \;}{\; \mu \; ! \;}
\Bigl\{
||-g_{(2k-1)\ell}(Z)||^{\mu_2}||-g^{\infty}_{(2k-1)\ell+1}(Z)||^{\mu_3}   \nn  \\
&\qquad \qquad \qquad \qquad 
+||-g_{(2k)\ell}(Z)||^{\mu_2}||-g^{\infty}_{(2k)\ell+1}(Z)||^{\mu_3}
\Bigr\}  \nn  \\
&<\sum^{t+1}_{k=1}\sum_{\substack{|\mu|=2\beta_{k} \\ \mu_1\not=2\beta_k}}
\dfrac{\; (2\beta_k)! \;}{\; \mu \; ! \;}
2\varepsilon^{\mu_2}_{\ell}.   \nn
\end{align}
Since $\sum \varepsilon^{1/2}_j$ converges, 
this goes to zero when $\ell$ goes to infinity. 
In both cases, the summation-term in (\ref{ellinfty}) goes to zero. 
Hence the fact that 
$\lim |||F_{\ell}(Z)|||^{2\beta}=1$ 
uniformly on $\pa E((m);(\alpha))$ 
implies that 
$|||F(Z)|||^{2\beta}$ tends to one when 
$Z \in E((m); (\alpha))$ goes towards the boundary. 
This means that $F(Z)$ is a proper holomorphic mapping.  
\end{proof}

\section {Construction of non-extendable proper holomorphic mapping}
In this section, we construct a proper holomorphic mapping between generalized complex pseudoellipsoids 
which does not extend continuously to the closure of the source domain. 

First, let $u$ be a continuous function on the circle $\pa \Delta$ and $\tilde{u}$ its 
harmonic extension to $\Delta$. 
Denote by $\tilde{v}$ and $v$ its harmonic conjugate and its boundary value. 
Then we can choose a continuous function $u$ such that the function $v$ is not continuous. 
Put $h(Z)=\iota(Z, \text{exp}(\tilde{u}(Z)+i\tilde{v}(Z)))$ for a sufficiently small $\iota>0$,  
then we can make the norm $||h(Z)||$ arbitrary small. 
Now we apply Proposition~{\ref{inductiveproposition}} to $h$ to obtain 
a mapping $F : E((m);(\alpha)) \to E((n);(\beta))$, which is a proper holomorphic mapping according to 
Theorem~{\ref{proper}} for $p=M_{s+1}+1$.  
Clearly, $F$ does not extend continuously to $\overline{E((m);(\alpha))}$, which is the desired mapping.


\begin{thebibliography}{99}

\bibitem[BC]{BC}S~.Bell, ~D.~Catlin, {\itshape Boundary regularity of proper holomorphic mappings} 
Duke~Math.~J.~49,\\~(1982) 385--396
\bibitem[D]{D}A.~Dor, {\itshape Proper holomorphic maps between balls in one co-dimension}, 
              Arkiv for Mathematik vol.~28, no.~1, (1990) 49--100
\bibitem[DF]{DF}K.~Diederich, ~J.~E.~Fornaess, {\itshape Boundary regularity of proper holomorphic mappings}, 
                   Invent.~Math.~67,~ (1982) 363--384
\bibitem[F]{F}F.~Forstneric, {\itshape Embedding strictly pseudoconvex domains into balls}, Trans.~Amer.~Math.~Soc.\\~295~(1986)~no.1,~347-368
\bibitem[H]{H}X.~Huang, {\itshape On the linearity problem for proper 
          holomorphic maps between balls in complex spaces of different dimensions}
          J.~Diff.~Geom.~51,~(1999)~13-33
\bibitem[G]{G}J.~Globevnik, {\itshape Boundary continuity of complete proper holomorphic maps}
             J.~Math.~Anal.\\~Appl.~424~(2015) 824--825 

\end{thebibliography}
\end{document}